\documentclass[a4paper,12pt]{article}
\usepackage[width=15cm]{geometry}

\def\cleft{\hbox{[\kern-.16em\hbox{[}}}
\def\cright{\hbox{]\kern-.16em\hbox{]}}}

\usepackage{amssymb,amscd,amsthm,latexsym}
\usepackage[all]{xy}

\newtheorem{defi}{Definition}[section]
\newtheorem{theo}[defi]{Theorem}
\newtheorem{lemma}[defi]{Lemma}
\newtheorem{prop}[defi]{Proposition}
\newtheorem{coro}[defi]{Corollary}

\begin{document}
\author{Dominique Bourn and James R.A. Gray}
\date{}
\title{Aspects of algebraic exponentiation}

\maketitle

\begin{abstract}
We analyse some aspects of the notion of \emph{algebraic exponentiation} introduced by the second author \cite{Gr} and satisfied by the category $Gp$ of groups. We show how this notion provides a new approach to the categorical-algebraic question of the centralization. We explore, in the category $Gp$, the unusual universal properties and constructions determined by this notion, and we show how it is the origin of various properties of this category.
\end{abstract}

\textbf{Introduction}\\

In \cite{Gr}, the second author observed, by means of very straightforward Kan extension arguments, that, in the category $Gp$ of groups, the change of base functor with respect to the fibration of points along any group homomorphism $h:X\rightarrow Y$:
$$h^*: Pt_YGp\rightarrow Pt_XGp$$
has a right adjoint, revealing, for the category $Gp$, a property having a certain analogy with the property, for a category $\mathbb E$, of being \emph{locally cartesian closed}, namely the property that the following change of base functor:
$$h^*:\mathbb E/Y\rightarrow \mathbb E/X$$
has a right adjoint for any map $h$. On the other hand, he showed moreover that, in the algebraic context of unital categories $\mathbb C$, the fact that the change of base functor along the terminal map:
$$\tau_X^*: \mathbb C=Pt_1\mathbb C\rightarrow Pt_X\mathbb C$$
has  right adjoint is related to the existence of some generalized notion of centralization.

\smallskip 

Now, the property of local cartesian closedness is very powerful and well known to be shared, for instance, by any elementary topos. It is not worth insisting on its significance. We shall develop some aspects of this new concept of algebraic exponentiation.

\smallskip 

In Section 1), we shall more deeply analyse the parallelism with the cartesian closedness and we shall strictly elucidate the relationship with the classical notion of centralizer, in such a way that, \emph{when a unital category $\mathbb C$ is regular, any change of base functor $\tau_X^*$, as above, has a right adjoint if and only if any subobject has a centralizer}, revealing that, behind the notions of centre and centralizer, there was an unexpected wider-ranging phenomenon of functorial nature. In Section 2) we shall show that, in the Mal'cev context, algebraic exponentiation along split epimorphisms allows us to extend the existence of centralizers from subobjects to equivalence relations; accordingly, when the category $\mathbb C$ is moreover exact, we get a Schreier-Mac Lane extension theorem, according to \cite{B16}. In Section 3) we shall investigate the stability properties of algebraic exponentiation and in particular we shall how, in the efficiently regular context, the existence of a right adjoint to: $$h^*: Pt_Y\mathbb C\rightarrow Pt_X\mathbb C$$ can be extended from split epimorphims to regular epimorphisms. In Section 4), in the stricter context of protomodular categories, we give a detailed description of some constructions determined by the algebraic exponentiation of all morphisms, and in particular we shall investigate two main consequences, namely \emph{strong protomodularity} (which guarantees, among other things, that the commutation of two equivalence relations $(R,S)$ is characterized by the commutation of their associated normal subobjects $(I_R,I_S)$ \cite{B25}) and \emph{peri-abelianness} (which is strongly related to the cohomology of groups \cite{B15}). These last points show us how some well identified particular properties of the category $Gp$ of groups originate from this algebraic exponentiation property. In this same category $Gp$, we shall explore in detail the very unusual universal properties and constructions involved in algebraic exponentiation. 

On the other hand we shall enlarge the list of examples (Section 2.1) to some categories of topological models, such as topological groups and topological rings, and to non-pointed categories such as the fibres $Grd_X$ above the set $X$ of the fibration $(\;)_0:Grd\rightarrow Set$ from groupoids to sets whose fibre $Grd_1$ above $1$ is nothing but the category $Gp$ itself.

\section{Cartesian closedness $\;$ and $\;$ algebraic cartesian  closedness}

\subsection{Slice categories and cartesian closedness}\label{pres}

Let $\mathbb E$ be any finitely complete category and $Y$ an object in $\mathbb E$. Any object $f:X\rightarrow Y$ in the slice category $\mathbb E/Y$ has a specific presentation as the domain of an equalizer of a split pair in $\mathbb E/Y$:
$$
\xymatrix@=30pt{
{X\;} \ar@{>->}[r]^{(f,1)}\ar[rd]_{f} & Y\times X \ar@<1ex>[r]^{s_0\times X} \ar@<-1ex>[r]_{Y\times (f,1)}\ar[d]_{p_Y} & Y\times (Y \times X) \ar[ld]^{p_Y} \ar[l]\\
 & Y 
              }
$$
(where the common splitting of the parallel pair is the map $p_0\times X$) which actually comes from a monad on the slice category $\mathbb E/Y$. This presentation can be extended to the category $Pt_Y\mathbb E$ of points above $Y$, namely to the split epimorphisms, in the following way:
$$
\xymatrix@=30pt{
{X\;} \ar@{>->}[r]^{(f,1)}\ar@<-1ex>[rd]_{f} & Y\times X \ar@<1ex>[r]^{s_0\times X} \ar@<-1ex>[r]_{Y\times (f,1)}\ar@<-1ex>[d]_{p_Y} & Y\times (Y \times X) \ar@<1ex>[ld]_{p_Y}  \ar[l]\\
 & Y \ar[lu]_s \ar[u]_{(1,s)} \ar@<-2ex>[ru]_{(s_0,s)}
              }
$$ 
We shall need later on the following collateral consequence: when the category $\mathbb E$ is pointed, we get the kernel of $f$, from the previous diagram, by the following equalizer:
$$
\xymatrix@=30pt{
{Ker f\;} \ar@{>->}[r]^{k_f} &  X \ar@<1ex>[r]^{\iota_X} \ar@<-1ex>[r]_{(f,1)} & Y \times X  \ar[l]
              }
$$

Now consider the change of base along the terminal map $\tau_Y:Y\rightarrow 1$:
$$\tau_Y^*:\mathbb E\longrightarrow \mathbb E/Y$$
According to our initial remark and because of the left exactness of right adjoints, the question of the existence of a right adjoint to $\tau_Y^*$ is reduced to the existence of cofree structures for the projections $p_Y:Y\times X \rightarrow Y$ in $\mathbb E/Y$; and this cofree structure is nothing but the exponential $X^Y$. In other words:\\

\noindent \emph{The functor $\tau_Y^*:\mathbb E\rightarrow \mathbb E/Y$ has a right adjoint if and only if the functor $Y\times -:\mathbb E\rightarrow \mathbb E$ has a right adjoint. The category $\mathbb E$ is cartesian closed if and only if any functor $\tau_Y^*$ has a right adjoint.}

\subsection{Fibration of points and algebraic cartesian closedness}

In an algebraic context, no such exponential does exist in general, among other things because of the existence of a zero object in the main instances, and consequently no such right adjoint functor to $\tau_Y^*$. However we have the possibility to consider the existence of a right adjoint to the ``change of base'' functor $\tau_Y^*:Pt_1\mathbb E\rightarrow Pt_Y\mathbb E$, i.e, here, ``only'' with respect to the points of $\mathbb E/Y$ and $\mathbb E$; and even more generally the existence of a right adjoint to the change of base functor $f^*:Pt_{\mathbb E}(Y)\rightarrow Pt_{\mathbb E}(X)$ for a any map $f:X\rightarrow Y$. 

This idea was first introduced by the second author \cite{Gr} who showed these right adjoints to $f^*$ do exist in the categories $Gp$ of groups and $R$-$Lie$ of Lie $R$-algebras, for any commmutative ring $R$ \cite{Gr1}.

The previous observation above concerning the equalizer presentation of any split epimorphism applies now for the functor:
$$\tau_Y^*:Pt_1\mathbb E\longrightarrow Pt_Y\mathbb E$$
The question of the existence of a right adjoint to $\tau_Y^*$ is then reduced to the existence of cofree structures for the split epimorphisms $(p_Y,(1,u)):Y\times X \rightarrow Y$ in $Pt_Y\mathbb E$ where $u:Y\rightarrowtail X$ is can be chosen to be a monomorphism.

We shall work now more specifically in the algebraic context of a \emph{unital} \cite{BB}, or even \emph{wealky unital} category \cite{MF} $\mathbb C$. Recall:
\begin{defi}
A category $\mathbb C$ is unital (resp. weakly unital) when it is pointed, is finitety complete, and is such that any pair of maps of the following form:
$$X \stackrel{(1_X,0)}{\rightarrowtail} X\times Y \stackrel{(0,1_Y)}{\leftarrowtail} Y$$
is jointly strongly epic (resp. jointly epic).
\end{defi}
Accordingly a finitely complete pointed category $\mathbb C$ is unital if and only if the supremum of the two previous subobjects is $1_{X\times Y}$.
In these contexts, the functor $\tau_Y^*$ becomes fully faithful. Recall also that there is then an intrinsic notion of commutation for any pair of maps with same codomain. We say that a pair $(f,g)$ commutes:
$$
\xymatrix@=15pt{
{X \;} \ar@{>->}[r]^{(1,0)} \ar[rd]_f & X\times Y \ar@{.>}[d]^{\phi} & {\; Y}\ar@{>->}[l]_{(0,1)} \ar[ld]^g\\
& Z
}
$$
when there exits a factorization $\phi$ (called the \emph{cooperator} of this pair), the uniqueness of $\phi$ making this existence a property of the pair $(f,g)$.
In these algebraic contexts, the meaning of the existence of a right adjoint to the functor $\tau_Y^*$ above can be made much more algebraically civilized:
\begin{prop}\label{centr}
Suppose $\mathbb C$ is a unital (resp. weakly unital) category. The functor $\tau_Y^*:Pt_1\mathbb E\longrightarrow Pt_Y\mathbb E$ admits a right adjoint $\Phi_Y$ if and only if any subobject $u: Y\rightarrowtail X$ with domain $Y$ admits a universal map $\zeta_u: Z[u] \rightarrow X$ commuting with it. This universal map $\zeta_u$ is necessarily a monomorphism.
\end{prop}
\proof
The universal property of $\zeta_u$ translates exactly the universal property of the cofree structure of the split epimorphism $(p_Y,(1,u)): Y\times X \rightleftarrows Y$ with respect to $\tau_Y^*$. Indeed, the natural transformation $\varepsilon: \tau_Y^*.\Phi_Y\Rightarrow Id$ is produced by a map in $Pt_Y\mathbb C$:
$$
\xymatrix@=15pt{
Y\times Z[u] \ar[rr]^{(p_Y,\phi)}\ar@<-1ex>[rd]_{p_Y} & & Y\times X \ar@<-1ex>[ld]_{p_Y}\\
& Y \ar[lu]_{\iota_Y} \ar[ru]_{(1,u)}
              }
$$
which makes $\phi:Y\times Z[u]\rightarrow X$ the \emph{cooperator} of the commuting pair $(u,\zeta_u)$, with $\zeta_u=\phi.\iota_{Z[u]}$.
Consider now the kernel equivalence of the map $\zeta_u$:
$$
\xymatrix@=30pt{
R[\zeta_u] \ar@<1ex>[r]^{p_0} \ar@<-1ex>[r]_{p_1} & Z[u] \ar[r]^{\zeta_u} & X 
              }
$$
The map $\zeta_u.p_i$ commutes with $u$ by means of the cooperator $\phi. (Y\times p_i)$. Its factorization through $\zeta_u$ being unique, we get $p_0=p_1$; and the map $\zeta_u$ is a monomorphism.
\endproof
Now, starting from any split epimorphism $(f,s)$, the cofree structure $\Phi_Y[f,s]$ is the equalizer of the following upper horizontal parallel pair induced by the pair $(s_0\times X,Y\times (f,1))$:
$$
\xymatrix@=20pt{
{\Phi_Y[f,s]\;} \ar@{>->}[r] \ar[d] & Z[s] \ar@<1ex>[r] \ar@<-1ex>[r] \ar[d]_{\zeta_s} & Z[(1,s)] \ar[d]^{\zeta_{(1,s)}} \\
{Ker f\;} \ar@{>->}[r]_{k_f} & X \ar@<1ex>[r]^{\iota_X} \ar@<-1ex>[r]_{(f,1)} & Y \times X 
             }
$$
Since the lower line is also an equalizer and the maps $(\zeta_{s},\zeta_{(1,s)})$ are monomorphisms, the left-hand side square is a pullback and  $\Phi_Y[f,s]=Ker f\cap Z[s]$. 
According to the previous proposition and the parallelism with cartesian closedness, we shall introduce the following:
\begin{defi}
A category $\mathbb C$ with products is said to be algebraically cartesian closed (a.c.c.) when any functor $\tau_Y^*:Pt_1\mathbb E\longrightarrow Pt_Y\mathbb E$ has a right adjoint.
\end{defi}
\noindent On the other hand, we have the quite classical:
\begin{defi}
Suppose $\mathbb C$ is a unital (resp. weakly unital) category. The centralizer of a subobject $u: Y\rightarrowtail X$ is the largest subobject commuting with it, i.e. the universal monomorphism commuting with $u$.
\end{defi}
So, \emph{when $\mathbb C$ is a unital category which is algebraically cartesian closed, any subobject $u$ has a centralizer $\zeta_u$}. When $\mathbb C$ is regular (as it is the case for any variety of universal algebras) the converse is true:
\begin{prop}\label{centreg}
Suppose $\mathbb C$ is a regular unital category. Then it is algebraically cartesian closed if and only if any subobject $u:Y\rightarrowtail X$ has a centralizer.
\end{prop}
\proof
Let $h:T\rightarrow X$ be any map commuting with $u$. Consider the canonical decomposition of $h$ through a regular epimorphism $T\stackrel{\rho}{\twoheadrightarrow} V\stackrel{\bar h}{\rightarrowtail} X$. Then, since $\rho$ is a regular epimorphism the pair $(u,\bar h)$ of subobjects does commute in $\mathbb C$; so $\bar h$, and thus $h$, factorizes through the centralizer $\zeta_u: Z[u] \rightarrowtail X$ of $u$ which therefore becomes also the universal map commuting with $u$.
\endproof

The main consequence of the algebraic cartesian closedness is a specific commutation of limits: the functor $\tau_Y^*$, having a right adjoint, preserves the colimits which exist in $\mathbb C$. In particular, when a direct sum exits in $\mathbb C$: $U \stackrel{\iota_U}{\rightarrow} U+V \stackrel{\iota_V}{\leftarrow} V$ the following square is a pushout in $\mathbb C/Y$ and thus in $\mathbb C$:
$$
\xymatrix@=20pt{
 Y\times U \ar[rr]^{Y\times \iota_U}  && Y\times(U+V) \\
 Y \ar[u]^{(1_Y,0)} \ar[rr]_{(1_Y,0)}  && Y\times V \ar[u]_{Y\times \iota_V}       }
$$

\subsection{Examples}

The unital category $Mon$ of monoids is unital and, as a variety of algebras, is exact and therefore regular. It has centralizers and thus is algebraically cartesian closed. More generally any unital variety of algebras with centralizers (as the categories $Gp$ of groups or $CRg$ of commutative rings) is algebraically cartesian closed. In the category $Gp$ of groups any split epimorphism $(f,s)$ above $Y$ is of the kind $Y\ltimes_{\psi}K \rightleftarrows Y$ where $\psi$ is the associated action of the group $Y$ on the group $K$. Then $\Phi_Y[f,s]$ is nothing but the subgroup $\{k\in K/ \forall y\in Y, \; ^yk=k\}$ of $K$ of the invariant elements under the action $\psi$.

\subsection{Reflection of commuting pairs}

There is, at the level of unital categories, a very simple result which will be of consequence later on.

\begin{prop}\label{refcom}
Let $\mathbb C$ be a unital (resp. weakly unital) category and $(\Gamma,\epsilon,\nu)$ a left exact comonad on it. Then the category $Coalg \Gamma$ of $\Gamma$-coalgebras is unital (resp. weakly unital) and the left exact forgetful functor $U: Coalg \Gamma \rightarrow \mathbb C$ reflects the commuting pairs.
\end{prop}
\proof
Since the comonad $(\Gamma,\epsilon,\nu)$ is left exact and $\mathbb C$ is finitely complete, so is the category $Coalg \Gamma$, which is morover pointed since so is $\mathbb C$. The forgetful functor $U: Coalg \Gamma \rightarrow \mathbb C$ is left exact. The category $Coalg \Gamma$ is unital (resp. weakly unital) since the functor $U$ is conservative (resp. faithful). Now, suppose that we have a pair of morphisms in $Coalg \Gamma$:
$$(X,\xi) \stackrel{f}{\longrightarrow} (Z,\zeta) \stackrel{g}{\longleftarrow} (Y,\upsilon)$$
whose image by $U$ is endowed with a cooperator $\phi$. We have to show that this map $\phi$ is actually a map of coalgebras, namely that the following quadrangle commutes:
$$
\xymatrix@=15pt{
  {X\;} \ar@{>.>}[r]^{(1,0)} & X\times Y \ar[d]_{\phi} \ar[rd]^{\Gamma \xi\times \Gamma \upsilon} & {\;Y} \ar@{>.>}[l]_{(0,1)}\\
& Z \ar[rd]_{\zeta} & \Gamma X \times \Gamma Y \ar[d]^{\Gamma \phi}  \\
&   &  \Gamma Z 
} 
$$
which can be done by composition with the two upper horizontal maps.
\endproof

\subsection{Strongly unital categories}

A unital category $\mathbb C$ is strongly unital \cite{B0}, when in addition, for any object $Y$, the change of base functor $\tau_Y^*:\mathbb C\rightarrow Pt_Y\mathbb C$ is \emph{saturated on subobjects}, namely such that any subobject $R\rightarrowtail \tau_Y^*(Z)$ is, up to isomorphism, the image by $\tau_Y^*$ of some subobject $S\rightarrowtail Z$. In this context, the idempotent comonad associated with the algebraic cartesian closedness has a specific property: 
\begin{prop}\label{mono}
Suppose $\mathbb C$ is a strongly unital category. Suppose the functor $\tau_Y^*:Pt_1\mathbb C\longrightarrow Pt_Y\mathbb C$ admits a right adjoint $\Phi_Y$. Then the natural transformation of the induced idempotent left exact comonad $\varepsilon_Y: \tau_Y^*.\Phi_Y\Rightarrow Id$ is monomorphic. Moreover any subobject $j :\bullet \rightarrowtail \bullet$ in $Pt_Y\mathbb C$ produces a pullback in $Pt_Y\mathbb C$:
$$
\xymatrix@=20pt{
{\;\tau_Y^*.\Phi_Y(\bullet)\;} \ar@{>->}[r]^{\tau_Y^*.\Phi_Y(j)} \ar[d]_{\varepsilon_Y\bullet} & {\tau_Y^*.\Phi_Y(\bullet)\;} \ar[d]^{\varepsilon_Y\bullet} \\
{\bullet\;} \ar@{>->}[r]_{j} & \bullet
             }
$$
\end{prop}
\proof
The comonad is idempotent even if $\mathbb C$ is only unital since the functor $\tau_Y^*$ is fully faithful. From the construction of $\Phi_Y[f,s]$, it is sufficient to prove the assertion for $Z[u]=\Phi_Y[p_Y,(1,u)]$. We observed that the natural map $\varepsilon_Y(p_Y,(1,u))$ is nothing but the map $Y\times Z[u] \stackrel{(p_Y,\phi)}{\rightarrow} Y\times X$, where the map $\phi$ is the cooperator of $u$ and $\zeta_u$. Now consider the following diagram in $\mathbb C$:
$$
\xymatrix@=15pt{
 Y \ar[r]^<<<<{\iota_Y}\ar@<-1ex>[rdd]_{(1,u)} & Y\times Z[u]\ar[dd]^<<<<{(p_Y,\phi)} & Z[u] \ar[l]_>>>>{\iota_{Z[u]}} \ar[ddl]^{(0,\zeta_u)}\\
 && \\
 & Y\times X  \ar[uul]_{p_Y}    }
$$
Then, according to Lemma 1.8.18 in \cite{BB} the vertical central map is a monomorphism if and only if so is the map $(0,\zeta_u)$, which is the case here. The end of the proof is a consequence of the following very general lemma:
\endproof

\begin{lemma}\label{satur}
Let $U: \mathbb E\rightarrow \mathbb F$ be a left exact fully faithful functor between finitely complete categories. Suppose moreover that $U$ has a right adjoint $G$ such that the natural transformation of the induced idempotent left exact comonad $\varepsilon: U.G\Rightarrow Id$ is monomorphic. Then the functor $U$ is saturated on subobjects if and only if, given any subobject $j:X' \rightarrowtail X$ in $\mathbb F$, the following square is a pullback in $\mathbb F$:
$$
\xymatrix@=20pt{
{U.G(X')\;} \ar@{>->}[r]^{U.G(j)} \ar[d]_{\varepsilon_{X'}} & {U.G(X)\;} \ar[d]^{\varepsilon_{X}} \\
{X'\;} \ar@{>->}[r]_{j} & X
             }
$$
\end{lemma}
\proof
Suppose the previous condition is satisfied. Given any subobject $j$, when $\varepsilon_{X}$ is an isomorphism, so is $\varepsilon_{X'}$, and $U$ is saturated on subobjects. Conversely suppose $U$ saturated on subobjects, and consider the pullback of $\varepsilon_{X}$ along $j$ and the induced monomorphic factorization $\eta$:
$$
\xymatrix@=20pt{
{U.G(X')\;} \ar@{>->}[rrd]^{U.G(j)} \ar[ddr]_{\varepsilon_{X'}} \ar@{>->}[dr]_{\eta}\\
& {P\;} \ar@{>->}[r]_{\bar j} \ar[d]_{\epsilon} & {U.G(X)\;} \ar[d]^{\varepsilon_{X}} \\
& {X'\;} \ar@{>->}[r]_{j} & X
             }
$$
Since $U$ is saturated on subobjects we can choose an object $P=U(T)$. Accordingly there is a map $\bar{\eta}:U(T)=P\rightarrow U.G(X')$ such that $\varepsilon_{X'}.\bar{\eta}=\epsilon$. Since $\varepsilon_{X'}$ and $\epsilon$ are monomorphisms, this $\bar{\eta}$ is the inverse of $\eta$. Therefore the square in question is a pullback.  
\endproof

\section{Mal'cev context}

We shall work now in the algebraic context of Mal'cev categories in the sense of \cite{CLP} and \cite{CPP}. One way of saying that a category $\mathbb C$ is a Mal'cev category is to say that any pointed fibre $Pt_Y\mathbb C$ is unital or, equivalently that any pointed fibre $Pt_Y\mathbb C$ is strongly unital, see \cite{B0}. Let us introduce the following:

\begin{defi}
A finitely complete category $\mathbb C$ is said to be fiberwise algebraically cartesian closed (f.a.c.c.) when every fibre $Pt_Y\mathbb C$ is algebraically cartesian closed. A morphism $h:X\rightarrow Y$ in $\mathbb C$ is said to be algebraically exponentiable when the change of base functor $h^*:Pt_Y{\mathbb C}\rightarrow Pt_X{\mathbb C}$ along $h$ admits a right adjoint. The category $\mathbb C$ is said to be locally algebraically cartesian closed (l.a.c.c) when any morphism $h$ is algebraically exponentiable.
\end{defi}

So, a category $\mathbb C$ is fiberwise algebraically cartesian closed if and only if, given any split epimorphism $(f,s):X\rightleftarrows Y$ the change of base functor $f^*:Pt_Y{\mathbb C}\rightarrow Pt_X{\mathbb E}$ has a right adjoint $\Phi_f$. When the category  $\mathbb C$ is a regular Malcev category, it is equivalent, according to Proposition \ref{centreg}, to saying that \emph{in any fibre $Pt_Y\mathbb C$ there exists centralizers of subobjects}. We get immediately:

\begin{prop}\label{mono2}
Suppose $\mathbb C$ is an fiberwise algebraically cartesian closed Mal'cev category. Let $(f,s):X\rightleftarrows Y$ be any split epimorphism in $\mathbb C$. Then the change of base functor $f^*:Pt_Y{\mathbb C}\rightarrow Pt_X{\mathbb E}$ is such that the natural transformation $\varepsilon_f: f^*.\Phi_{f}\Rightarrow Id$ of the induced idempotent left exact comonad is monomorphic. Moreover any subobject $j :\bullet \rightarrowtail \bullet$ in $Pt_X\mathbb C$ produces a pullback in $Pt_X\mathbb C$:
$$
\xymatrix@=20pt{
{f^*.\Phi_f(\bullet)\;} \ar@{>->}[r]^{f^*.\Phi_f(j)} \ar[d]_{\varepsilon_f\bullet} & {f^*.\Phi_f(\bullet)\;} \ar[d]^{\varepsilon_f\bullet} \\
{\bullet\;} \ar@{>->}[r]_{j} & \bullet
             }
$$
\end{prop}
\proof
As we recalled above, the category $\mathbb C$ being Mal'cev, any fibre $Pt_Y\mathbb C$ is not only unital but also strongly unital. Accordingly, just apply Proposition \ref{mono}.
\endproof  

\subsection{Examples} 
In \cite{Gr} and\cite{Gr1}, it was shown that: the category $CRg$ of commutative rings is fiberwise algebraically cartesian closed but not locally algebraically cartesian closed; the categories $Gp$ of groups and $R$-$Lie$ of Lie $R$-algebras, for any commmutative ring $R$, are locally algebraically cartesian closed; when a category $\mathbb E$ is a cartesian closed category with pullbacks, the category $Gp\mathbb E$ of internal groups in $\mathbb E$ is locally algebraically cartesian closed. On the other hand a category $\mathbb A$ was defined as essentially affine \cite{B-1} when any change of base functor $h^*:Pt_Y{\mathbb A}\rightarrow Pt_X{\mathbb A}$ is an equivalence of categories; accordingly any essentially affine category is locally algebraically cartesian closed. In particular any additive category is locally algebraically cartesian closed.

\medskip

\noindent\textbf{Non-pointed examples 1: slice and coslice categories}
 
\begin{lemma}
Let $U:\mathbb C\rightarrow \mathbb D$ be a pullback preserving functor. Suppose moreover that it is a discrete fibration (resp. discrete cofibration). When $f:X\rightarrow Y$ is algebraically exponentiable in $\mathbb C$, so is $U(f)$ in $\mathbb D$. 
\end{lemma}
\proof Straightforward. \endproof

Let $\mathbb C$ be any category. Then, for any object $Y$ in $\mathbb C$, the domain functor $\mathbb C/Y\rightarrow \mathbb C$ is a discrete fibration which preserves pullbacks. Still, for any object $Y$ in $\mathbb C$, the codomain functor $Y/\mathbb C\rightarrow \mathbb C$ is a left exact discrete cofibration. Accordingly fiberwise algebraic cartesian closedness (resp. locally algebraic cartesian closedness) is stable under slicing and coslicing, giving rise to non-pointed examples. As a consequence, when $\mathbb C$ is fiberwise algebraically cartesian closed (resp. locally algebraically cartesian closed), so is any fibre $Pt_Y\mathbb C$, which is the coslice category on the terminal object of the slice category $\mathbb C/Y$.

\smallskip

\noindent\textbf{Non-pointed examples 2: the fibres of the fibration $Grd\rightarrow Set$} 

Let us denote by $Grd$ the category of groupoids and by $(\,)_0:Grd\rightarrow Set$ the forgetful functor associating with any groupoid $\underline Y_1$ the set $Y_0$ of its objects; it is a fibration whose cartesian maps in $Grd$ are the fully faithful functors. The fibre above $1$ is clearly the pointed category $Gp$ of groups. We shall denote by $Grd_X$ the fibre above the set $X$: its objects are the groupoids whose set of objects is $X$ and its arrows are those functors between such groupoids which are bijective on objects. We know that these fibres $Grd_X$ are protomodular \cite{B-1} and thus Mal'cev categories, and they are no longer pointed. The aim of this section is to show that any fibre $Grd_X$ is locally algebraically cartesian closed; the proof will be a slight generalization of the proof for $Gp$.

\begin{lemma}
Let be given a groupoid $\underline Y_1$. The fibre $Pt_{\underline Y_1}(Grd_{Y_0})$ is in bijection with the functor category $\mathcal F(\underline Y_1,Gp)$. Suppose $\underline F_1:\underline Y_1\rightarrow \underline Z_1$ be any functor. Then the change of base functor $\underline F_1^*:Pt_{\underline Z_1}(Grd_{Z_0})\rightarrow Pt_{\underline Y_1}(Grd_{Y_0})$ is naturally isomorphic to the functor $\mathcal F(\underline F_1,Gp): \mathcal F(\underline Z_1,Gp)\rightarrow \mathcal F(\underline Y_1,Gp)$.
\end{lemma}
\proof
The category $Gp$ can be considered as the full subcategory of the category $Cat$ (of categories) whose objects are the groupoids with only one object. The lemma is a specification of the Grothendieck construction. From any functor $H:\underline Y_1 \rightarrow Gp$ we get a bijective on objects split cofibration $\underline H_1:\underline X_1\rightarrow \underline Y_1$ where a map $y\rightarrow y'$ in $\underline X_1$ is a pair $(f,\gamma)$ with $f:y\rightarrow y'$ is a map in $\underline Y_1$ and $\gamma \in H(y')$. The composition is defined by: $(f',\gamma ').(f,\gamma) =(f'.f,\gamma'.H(f')(\gamma))$. The functor $\underline H_1$, defined by $\underline H_1(f,\gamma)=f$, has a splitting $\underline T_1$ defined by $\underline T_1(f)=(f,1_{H(y')})$. Conversely any split bijective on objects functor $\underline H_1:\underline X_1\rightarrow \underline Y_1$ is necessarily a split cofibration and determines a functor $H:\underline Y_1 \rightarrow Gp$. The end of the proof is straightforward.
\endproof

\begin{theo}
Consider the fibration $(\,)_0:Grd\rightarrow Set$; any of its fibres $Grd_X$ is locally algebraically cartesian closed.
\end{theo}
\proof
Given any functor $\underline F_1:\underline Y_1\rightarrow \underline Z_1$ between two groupoids, the functor $\mathcal F(\underline F_1,Gp): \mathcal F(\underline Z_1,Gp)\rightarrow \mathcal F(\underline Y_1,Gp)$ admits a right adjoint, given by the right Kan extension along the functor $\underline F_1$. Then, according to the previous lemma, any change of base functor $\underline F_1^*:Pt_{\underline Z_1}(Grd_{Z_0})\rightarrow Pt_{\underline Y_1}(Grd_{Y_0})$ has a right adjoint. The theorem holds by taking $\underline F_1$ a bijective on objects functor with $Z_0=X=Y_0$.
\endproof

\smallskip

\noindent\textbf{Topological models}

In this section we shall make explicit some topological examples. Let $\mathbb T$ be a Mal'cev theory, $\mathbb V(\mathbb T)$ the corresponding variety of $\mathbb T$-algebras and $Top(\mathbb T)$ the category of topological $\mathbb T$-algebras. Recall that $Top(\mathbb T)$ is then a regular Mal'cev category, see\cite{JP}, whose regular epimorphisms are the open surjective maps and recall also the following:
\begin{lemma}
Let $\mathbb T$ be a Mal'cev theory. Then the forgetful left exact functor $U:Top(\mathbb T)\rightarrow \mathbb V(\mathbb T)$ reflects the pullback of split epimorphisms along regular epimorphisms.
\end{lemma}
From this, we get:
\begin{prop}
Let $\mathbb T$ be a Mal'cev theory such that $\mathbb V(\mathbb T)$ is fiberwise algebraically cartesian closed. Then the category $Top(\mathbb T)$ of topological $\mathbb T$-algebras is fiberwise algebraically cartesian closed.
\end{prop}
\proof
Let $(f,s):X\rightleftarrows Y$ be a split epimorphism in $Top(\mathbb T)$ and $(g,t):V\rightleftarrows X$ an object in $Pt_XTop(\mathbb T)$. First consider the following diagram given by the algebraic exponentiation in $\mathbb V(\mathbb T)$:
$$
\xymatrix@=12pt{ 
 && {\; U(f)^*\Phi_{U(f)}[U(g),U(t)]\;} \ar@{>->}[lld]_{\varepsilon_{d_1}(U(g),U(t))} \ar@{->>}[rr]_{\phi} \ar[lldd]  && {\;\Phi_{U(f)}[U(g),U(t)]}\ar@<-1ex>@{>->}[ll]_{\sigma} \ar[lldd]_{\gamma} \\
U(V) \ar[d]^{U(g)} &&  \\
U(X) \ar@<1ex>[u]^{U(t)} \ar@{->>}[rr]^{U(f)} \ar@<-1ex>[uurr]_{\beta} && {\; U(Y)} \ar@<1ex>@{>->}[ll]^{U(s)} \ar@<-1ex>[uurr]_{\tau}  
 }
$$
and then the following one in $Top(\mathbb T)$ where $V'$ is the algebra $U(f)^*\Phi_{U(f)}[U(g),U(t)]$ equipped with the topology induced by the one on $V$:
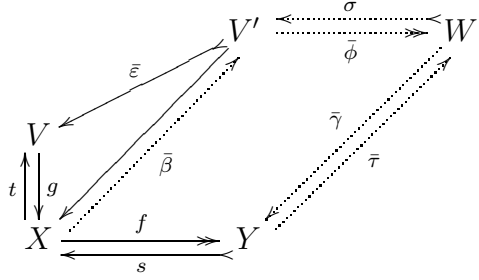
$$
\xymatrix@=25pt{ 
 && {\; V'\;} \ar@{>->}[lld]_{\bar{\varepsilon}} \ar@{.>>}[rr]_{\bar{\phi}} \ar[lldd]  && {\;W}\ar@<-1ex>@{>.>}[ll]_{\bar{\sigma}} \ar@{.>}[lldd]_{\bar{\gamma}} \\
V \ar[d]^{g} &&  \\
X \ar@<1ex>[u]^{t} \ar@{->>}[rr]^{f} \ar@{.>}@<-1ex>[uurr]_{\bar{\beta}} && {\; Y} \ar@<1ex>@{>->}[ll]^{s} \ar@{.>}@<-1ex>[uurr]_{\bar{\tau}}  
 }
$$
The map $t$ in $Top(\mathbb T)$ and the factorization $\beta$ in $\mathbb V(\mathbb T)$ produce the factorization $\bar{\beta}$. Then put the quotient topology on $\Phi_{U(f)}[U(g),U(t)]$ to produce the object $W$ in $Top(\mathbb T)$. Then we get the dotted maps above the quadrangle pullback of our initial diagram. The previous lemma asserts that it is a pullback. From this situation, it is straigforward to check that the split epimorphism $(\bar{\gamma},\bar{\tau})$ has the desired universal property with respect to the change of base functor $f^*$.
\endproof

Accordingly the categories $TopGp$ and $TopCRg$ of topological groups and topological commutative rings are fiberwise algebraically cartesian closed.

\subsection{Abelian split extension}

A split epimorphism $(f,s):X \rightleftarrows Y$ is said to be abelian in a Mal'cev category $\mathbb C$ when it is an abelian object in the fibre $Pt_Y\mathbb C$. Since any right adjoint functor is left exact, any algebraically exponentiable map $h:Y\rightarrow Y'$ is such that the restriction of $\Phi_h:Pt_Y\mathbb C\rightarrow Pt_{Y'}\mathbb C$ to the abelian objects determines a functor:
$$\Phi_h:AbPt_Y\mathbb C\rightarrow AbPt_{Y'}\mathbb C$$
In particular, when $\mathbb C$ is pointed and fiberwise algebraically cartesian closed, when $(f,s)$ is abelian, so is the object $\Phi_Y[f,s]$. Recall that, when $\mathbb C$ is the category $Gp$ of groups, a split epimorphism is abelian if and only if it has an abelian kernel $A$:
$$
\xymatrix@=30pt{
1 \ar[r] & {A\;} \ar@{>->}[r] & Y\ltimes_{\psi} A \ar@{->>}[r]^{\pi} & {\;Y} \ar[r] \ar@<1ex>@{>->}[l]^{\sigma} & 1        }
$$
The (evidently abelian) subgroup $\Phi_Y[\pi,\sigma]$ of the invariant elements of $A$ under the action $\psi$ was denoted $A^Y$ in \cite{ML} and shown to be the $0$-dimensional cohomology group $H^0_{\psi}(Y,A)$. This fact was used to introduce in the Mal'cev context a notion of internal cohomology in \cite{Gr0}.

\subsection{Centralizer of equivalence relations}

In the Mal'cev context, there exits also an intrinsic notion of commutation at the level of equivalence relations, see \cite{BG2}. First, the subobjects of the object $(p_0,s_0):X\times X\rightleftarrows X$ in the fibre $Pt_X\mathbb C$ coincide exactly with the reflexive relations on $X$, hence, in the Mal'cev context, with the equivalence relations on $X$. Recall that two equivalence relations $R$ and $S$ on an object $X$ commute in $\mathbb C$ if and only if the two following subobjects in the fibre $Pt_X\mathbb C$ do commute in $Pt_X\mathbb C$, see Proposition 2.6.12 in \cite{BB}:
$$
\xymatrix@=30pt{
& & & {\;S} \ar[ddll]_{d_0} \ar[dll]_{(d_0,d_1)} \\
{R\;} \ar@{>->}[r]^{(d_1,d_0)}\ar@<-1ex>[rd]_{d_1} & X\times X \ar@<-1ex>[d]_{p_0} \\
 & X \ar[lu]_{s_0} \ar[u]_{s_0} \ar@<-1ex>[rruu]_{s_0}
              }
$$
the choice of this presentation ($R^{op}$ rather than $R$) being made for technical reasons related to the classical presentation of the axioms of a Mal'cev operation. So, in the fiberwise algebraically cartesian closed context, the existence of centralizers can be immediately transfered to the level of equivalence relations.
 
\begin{prop}\label{centrequ}
Suppose $\mathbb C$ is a Mal'cev category which is fiberwise algebraically cartesian closed. Let $R$ be any equivalence relation on the object $X$. Then the centralizer $Z(R)$ of the equivalence relation $R$ does exist in $\mathbb C$ and is given by the domain of $\Phi_{d_1}[p_R,(1,d_0)]$:
$$
\xymatrix@=20pt{ 
 && {\; \bullet\;} \ar@{>->}[lld]_{\varepsilon_{d_1}(p_R,(1,d_0))} \ar[rr] \ar[lldd] && {\;Z(R)}\ar@<-1ex>@{>->}[ll] \ar[lldd] \ar@{>->}[lld]_>>>>{(d_0,d_1)}\\
R\times X \ar[d]^{p_R} && {\; X\times X} \ar@{>->}[ll] \ar[d]^{p_0} \\
R \ar@<1ex>[u]^{(1,d_0)} \ar[rr]^{d_1} \ar@<-1ex>[uurr] && {\; X} \ar@<1ex>@{>->}[ll]^{s_0} \ar@<-1ex>[uurr] \ar@<1ex>[u]^{s_0} 
 }
$$
\end{prop}
\proof
Since the category $\mathbb C$ is fiberwise algebraically cartesian closed, the unital fibre $Pt_X\mathbb C$ has centralizers of subobjects, and according to the previous recall about equivalence relations and their commutations, the centralizer of $R$ in $\mathbb C$ is nothing but the centralizer of the following subobject in the fibre $Pt_X\mathbb C$: 
$$
\xymatrix@=30pt{
{R\;} \ar@{>->}[r]^{(d_1,d_0)}\ar@<-1ex>[rd]_{d_1} & X\times X \ar@<-1ex>[d]_{p_0} \\
 & X \ar[lu]_{s_0} \ar[u]_{s_0} 
              }
$$
According to Proposition \ref{centr}, it is given by $\Phi_{d_1}[p_R,(1,d_0)]$.
\endproof

So, when the category $\mathbb C$ is exact, Mal'cev and fiberwise algebraically cartesian closed, the existence of centralizers makes the Schreier- Mac Lane extensions classification theorem hold, see \cite{B16}.

\section{Some properties of the algebraic exponentiable morphisms}

\subsection{Stability under pullback along split epimorphisms}

We show first that the algebraically exponentiable morphisms are stable under pullback along split epimorphisms. It is a consequence of the following very general lemma:
\begin{lemma}
Let $\mathbb E$ be a category with pullbacks and $U:\mathbb E\rightarrow \mathbb F$ a functor which admits a right adjoint $G$. Then, for any object $X\in \mathbb E$, the induced functor:
$$
\xymatrix@=20pt{ U_X: Pt_X\mathbb E \ar[r] & Pt_{UX}\mathbb F}
$$
has a right adjoint $G_X$. When moreover the category $\mathbb F$ has pullbacks, any map $f:X\rightarrow X'$ makes the following leftward diagram commute up to a natural isomorphism:
$$
\xymatrix@=20pt{
 Pt_{X'}\mathbb E \ar[d]_{f^*} \ar@<-1ex>@{.>}[r]_{U_{X'}} & \ar[l]_{G_{X'}} Pt_{UX'}\mathbb F \ar[d]^{Uf^*}\\
  Pt_X\mathbb E \ar@<1ex>@{.>}[r]^{U_X} & \ar[l]^{G_X} Pt_{UX}\mathbb F
  }
$$
When, in addition, $U$ is left exact the previous diagram also commutes at the level of doted arrows.
\end{lemma}
\proof
Let $(\tau,\sigma):T\rightleftarrows UX$ be an object of $Pt_{UX}\mathbb F$. It is straighforward to check that $G_X(\tau,\sigma)$ is given by the following pullback in $\mathbb E$, where $\eta_X$ is the unit of the adjunction:$$
\xymatrix@=20pt{
 \bullet \ar[d] \ar[r] & GT \ar[d]_{G\tau}\\
  X \ar[r]_{\eta_X} \ar@<-1ex>[u] & GUX \ar@<-1ex>[u]_{G\sigma}
  }
$$
The second point is a consequence of the naturality of the unit $\eta$ and of the fact that the right adjoint functor $G$ preserves pullbacks. From that, the last point is straightforward.
\endproof
Whence the following:
\begin{prop}\label{pulst}
Let $\mathbb C$ be a finitely complete category. Then the algebraically exponentiable morphisms in $\mathbb C$ are stable under pullback along split epimorphisms. Moreover any pullback in the category $\mathbb C$ with $y$ algebraically exponentiable:
$$
\xymatrix@=20pt{
  X \ar[d]_f \ar[r]^x & X' \ar[d]_{f'}\\
  Y \ar[r]_{y} \ar@<-1ex>[u]_s & Y' \ar@<-1ex>[u]_{s'}
  }
$$
satisfies the following Beck-Chevalley conditions, i.e. makes the following diagrams commute up to natural isomorphisms:
$$
\xymatrix@=20pt{
 Pt_{X}\mathbb C  \ar[r]^{\Phi_x} & Pt_{X'}\mathbb C \ar@<1ex>@{.>}[l]^{x^*} &&  Pt_{X}\mathbb C  \ar[r]^{\Phi_x} \ar[d]_{s^*} & Pt_{X'}\mathbb C \ar[d]^{s'^*} \ar@<1ex>@{.>}[l]^{x^*}\\
  Pt_{Y}\mathbb C \ar[u]^{f^*} \ar[r]_{\Phi_y} & Pt_{Y'}\mathbb C \ar@<-1ex>@{.>}[l]_{y^*}  \ar[u]_{f'^*} && Pt_{Y}\mathbb C  \ar[r]_{\Phi_y} & Pt_{Y'}\mathbb C \ar@<-1ex>@{.>}[l]_{y^*} 
  }
$$
\end{prop}
\proof
Apply the previous lemma to the functor $y^*:Pt_{Y'}\mathbb C\rightarrow Pt_Y\mathbb C$ and notice that we have $Pt_{(f',s')}(Pt_{Y'}\mathbb C)=Pt_{X'}\mathbb C$ and $Pt_{(f,s)}(Pt_{Y}\mathbb C)=Pt_{X}\mathbb C$. Then consider the following morphisms in $Pt_{Y'}\mathbb C$:
$$
\xymatrix@=10pt{
  X' \ar@<-1ex>[dr]_{f'} \ar[rr]^{f'} && Y' \ar[dl]_{1} && Y' \ar@<-1ex>[dr]_1 \ar[rr]^{s'} && X' \ar[dl]_<<<<<{f'} \\
   &  Y' \ar@<-1ex>[ur]_{1}  \ar[ul]_{s'}      &      &&   &  Y' \ar@<-1ex>[ur]_{s'}  \ar[ul]_1
} 
$$
\endproof
Then we get immediately the following:
\begin{coro}\label{splitalg}
When a split epimorphism $(f,s) :X\rightleftarrows Y$ is algebraically exponentiable, the induced endofunctor $f^*.\Phi_{f}$ on $Pt_X\mathbb C$ is (up to a natural isomorphism) equal to the endofunctor $\Phi_{p_0}.p_1^*$, where $p_0$ and $p_1$ are given by the kernel equivalence relation:
$$\xymatrix@=10pt
{
R[f]  \ar@<-2ex>[rr]_{p_{1}} \ar@<1ex>[rr]^{p_{0}}  && X \ar[ll]^{s_0} \ar@{->>}[rr]_{f} && Y
}
$$
This endofunctor $\Phi_{p_0}.p_1^*$ on $Pt_X\mathbb C$ inherits the left exact comonad structure induced by the adjoint pair $(f^*,\Phi_{f})$.
\end{coro}

\subsection{The efficiently regular context}

A regular category $\mathbb C$ is said to be efficiently regular, when, in addition, any equivalence relation $S$, on an object $X$, which is included in an effective equivalence relation $S \stackrel{m}{\rightarrowtail} R[f]$ by an effective monomorphism $m$, is itself effective. The main examples are the categories $TopGp$ and $TopAb$ of topological groups and abelian groups. Any exact category is efficiently regular. When the category $\mathbb C$ is moreover efficiently regular, we can extend algebraic exponentiability from split epimorphisms to regular epimorphisms and have a kind of converse to Proposition \ref{pulst}. For that, let us begin by the following:
\begin{prop}
Let $\mathbb C$ be efficiently regular. Consider an internal discrete cofibration: $\underline f_1:\underline X_1\rightarrow \underline Y_1$ between two groupoids:
$$\xymatrix@=7pt
{
R[d_0] \ar@<-3ex>[rr]_{d_{2}^1} \ar[rr]_{d_{1}^1} \ar@<1ex>[rr]^{d_{0}^1} \ar[ddd]_{R(f_1)} & & X_1  \ar@<-2ex>[rr]_{d_{1}} \ar@<1ex>[rr]^{d_{0}} \ar[ddd]_{f_1}  && X_0 \ar[ddd]^{f_0} \ar[ll]^{s_0} \\
&&&&\\
&&&&\\
R[d_0] \ar@<-3ex>[rr]_{d_{2}^1} \ar[rr]_{d_{1}^1} \ar@<1ex>[rr]^{d_{0}^1} & & Y_1  \ar@<-2ex>[rr]_{d_{1}} \ar@<1ex>[rr]^{d_{0}}  && Y_0 \ar[ll]^{s_0}
}
$$
Suppose the morphism $f_0$ is algebraically exponentiable. Then the functor $\underline f_1^*: SCof_{\underline Y_1} \rightarrow SCof_{\underline X_1}$
from the split discrete cofibrations above $\underline Y_1$ to the split discrete cofibrations above $\underline X_1$ defined by pulling back along the functor $\underline f_1$ admits a right adjoint which is constructed levelwise.
\end{prop}
\proof
According to Proposition \ref{pulst}, since the vertical square with the $d_0$ is a pullback (the functor $\underline f_1$ being a discrete cofibration), the maps $f_1$ and $R(f_1)$ are also algebraically exponentiable. Let $(\underline{\alpha}_1,\underline{\beta}_1): \underline{T}_1\rightleftarrows \underline{X}_1$ be a split discrete fibration above $\underline{X}_1$. We are going to show that the split epimorphisms $\Phi_{f_0}(\alpha_0,\beta_0)=(\bar{\alpha}_0,\bar{\beta}_0):W_0\rightleftarrows Y_0$ and $\Phi_{f_1}(\alpha_1,\beta_1)=(\bar{\alpha}_1,\bar{\beta}_1):W_1\rightleftarrows Y_1$ are actually underlying a discrete cofibration above $\underline{Y}_1$, which will determine the construction of the levelwise right adjoint in question. For that, let us consider the following diagram:
$$\xymatrix@=7pt
{
 &&&  T_1  \ar@<-1ex>[rr]_{} \ar@<1ex>[rr]^{d_{0}} \ar@<-1ex>[dlll]_{\alpha_1}  && T_0 \ar@<1ex>[dlll]^{\alpha_0} \ar[ll]^{}
&&&&\\
 X_1  \ar@<-1ex>[rr]_{d_{1}} \ar@<1ex>[rr]^>>{d_{0}} \ar[ddd]_<<<<<{f_1}  && X_0 \ar[ddd]^<<<<<{f_0} \ar[ll]^{} \\
&&&&\\
 &&&  W_1  \ar@<-1ex>@{.>}[rr]_{} \ar@<1ex>[rr]^{d_{0}} \ar@<-1ex>[dlll]_{\bar{\alpha}_1}  && W_0 \ar@<1ex>[dlll]^{\bar{\alpha}_0} \ar[ll]^{}
&&&&\\
Y_1  \ar@<-1ex>[rr]_{d_{1}} \ar@<1ex>[rr]^>>{d_{0}}  && Y_0 \ar[ll]^{}
}
$$
Since the square with the $d_0$ in the statement is a pullback underlying a pullback of split epimorphisms, then, according to Proposition \ref{pulst}, the Beck-Chevalley condition holds for this square. Consequently the lower quadrangle with $d_0$ above is underlying a pullback of split epimorphisms. But a discrete cofibration between groupoids is also a discrete fibration and the square with the $d_1$ in the statement is also a pullback. Moreover, the Beck-Chevalley condition not only says that the co-free object are preserved by pullbacks, but also the universal natural transformation $f_i^*.\Phi_{f_i}\Rightarrow 1_{Pt_{Y_i}\mathbb C}$, $i\in \{0,1\}$. This determines an arrow $d_1:W_1 \rightarrow Y_1$ which makes also the lower quadrangle with $d_1$ a pullback, and produces a reflexive graph $W_1\rightrightarrows W_0$. The same Beck-Chevalley condition makes this reflexive graph underlying a groupoid structure and $\underline{\alpha}_1$ a discrete fibration which is, by construction, a levelwise cofree structure with respect to the pulling back along the functor $\underline f_1$.
\endproof
Whence the following ``converse'' to Proposition \ref{pulst}:
\begin{prop}
Let $\mathbb C$ be an efficiently regular. Consider the following pullback with $f'$ a regular epimorphism:
$$
\xymatrix@=20pt{
 X \ar@{->>}[r]^{f} \ar[d]_{x} & Y \ar[d]^{y}\\
 X'\ar@{->>}[r]_{f'}  & Y' 
 }
$$
Then, when the morphism $x$ is algebraically exponentiable, so is the morphism $y$, and we have the Beck-Chevalley commutation:
$$
\xymatrix@=20pt{
 Pt_{X} \ar[d]_{\Phi_x}  & Pt_{Y} \ar[l]_{f^*} \ar[d]^{\Phi_y}\\
 Pt_{X'}  & Pt_{Y'} \ar[l]^{f'^*} 
 }
$$ 
\end{prop}
\proof
Complete the previous pullback by the following diagram:
$$\xymatrix@=7pt
{
R^2[f] \ar@<-3ex>[rr]_{p_{2}^1} \ar[rr]_{p_{1}^1} \ar@<1ex>[rr]^{p_{0}^1} \ar[ddd]_{R^2(x)} & & R[f]  \ar@<-2ex>[rr]_{p_{1}} \ar@<1ex>[rr]^{p_{0}} \ar[ddd]_{R(x)}  && X_0 \ar[ddd]_{x} \ar[ll]^{s_0} \ar@{->>}[rr]^{f}&& Y \ar[ddd]^{y} \\
&&&&\\
&&&&\\
R^2[f'] \ar@<-3ex>[rr]_{p_{2}^1} \ar[rr]_{p_{1}^1} \ar@<1ex>[rr]^{p_{0}^1} & & R[f']  \ar@<-2ex>[rr]_{p_{1}} \ar@<1ex>[rr]^{p_{0}}  && X' \ar[ll]^{s_0} \ar@{->>}[rr]_{f'}&& Y'
}
$$
which determines a discrete cofibration $\underline R_1(x): \underline R_1[f]: \underline R_1[f']$ between the left hand side induced horizontal groupoids. According to the previous proposition the change of base functor $\underline R_1(x)^*: SCof_{\underline R_1[f']} \rightarrow SCof_{\underline R_1[f]}$ admits a right adjoint. Now consider the following commutative square:
$$
\xymatrix@=20pt{
 Pt_{Y'}\mathbb C \ar[r]^{y^*} \ar[d]_{F_{Y'}} & Pt_{Y}\mathbb C \ar[d]^{F_Y}\\
 {SCof_{\underline R_1[f']}\;} \ar[r]_{\underline R_1(x)^*}  & SCof_{\underline R_1[f]}
 }
$$
where the functors $F_Y$ and $F_{Y'}$ are the canonical straighforward functors which are fully faithful since $f$ and $f'$ are regular epimorphisms. They are also essentially surjective, since, in an efficiently regular category, any equivalence fibration which is discretely cofibered above an effective equivalence relation is itself effective. Accordingly the functors $F_Y$ and $F_{Y'}$ are equivalences of categories, and the change of base functor $y^*$ admits a right adjoint. This construction of the right adjoint to $y^*$ imposes the Beck-Chevalley condition. 
\endproof
\begin{coro}\label{regexp}
Let $\mathbb C$ be efficiently regular and $f:X\twoheadrightarrow Y$ a regular epimorphism such that the map $p_0$ below:
$$\xymatrix@=10pt
{
R[f]  \ar@<-2ex>[rr]_{p_{1}} \ar@<1ex>[rr]^{p_{0}}  && X \ar[ll]^{s_0} \ar@{->>}[rr]_{f} && Y
}
$$
 is algebraically exponentiable, then $f$ is itself algebraically exponentiable and we have:  $f^*.\Phi_{f} \simeq \Phi_{p_0}.p_1^*$.

When $\mathbb C$ is efficiently regular and fiberwise algebraically cartesian closed, then any regular epimorphism $f:X\twoheadrightarrow Y$ is algebraically exponentiable. 
\end{coro}
\proof
Consider the following pullback:
$$
\xymatrix@=20pt{
 R[f] \ar@{->>}[r]^{p_1} \ar[d]_{p_0} & x \ar[d]^{f}\\
 X\ar@{->>}[r]_{f}  & Y 
 }
$$
and apply the previous proposition.
\endproof

\section{Protomodular context}

We shall now work in the stronger context of a protomodular category $\mathbb C$ \cite{B-1}, which means that any (left exact) change of base functor:
$$h^*: Pt_Y\mathbb C\rightarrow Pt_X\mathbb C$$
is conservative. We get immediately the following:
\begin{prop}
Suppose $\mathbb C$ is protomodular, then any any change of base functor $h^*$ along an algebraic exponentiable map $h:X\rightarrow Y$ reflects commuting pairs. 
\end{prop}
\proof
In the protomodular context, any change of base functor being left exact and conservative, any algebraic exponentiable map $h:X\rightarrow Y$ makes this functor $h^*$ immediately  comonadic \cite{Gr}. Accordingly, the assertion in question is a direct consequence of Proposition \ref{refcom}. 
\endproof

\subsection{Locally algebraically cartesian closed pointed protomodular categories}

On the one hand, in \cite{BJK0} and \cite{BJK}, the notion of \emph{action representative} category was introduced, i.e. a pointed protomodular category $\mathbb C$ in which each object $X$ admits a universal split extension with kernel $X$ (=split extension classifier):
$$
\xymatrix@=20pt
{ 
  {X\;} \ar@{>->}[r]^{\gamma} & D_1(X) \ar@{->>}[r]^{d_0} & {\;D(X)} \ar@<+1,ex>@{>->}[l]^{s_0} 
 }
$$
in the sense that any other split extension with kernel $X$ determines a unique morphism $\chi : G\rightarrow D(X)$ such that the following diagram commutes and the right hand side squares are pullbacks:
$$
\xymatrix@=20pt
{ 
  {X\;} \ar@{>->}[r]^{k} \ar[d]_{1_X} &  H  \ar[d]_{\chi_1} \ar@{->>}[r]^{f} & {\;G} \ar[d]^{\chi} \ar@<+1,ex>@{>->}[l]^{s} \\
  {X\;} \ar@{>->}[r]_{\gamma} & D_1(X) \ar@{->>}[r]^{d_0}  & {\;D(X)} \ar@<+1,ex>@{>->}[l]^{s_0}
 }
$$  

On the other hand, in \cite{Gr}, the second author showed that \emph{when the category $\mathbb C$ is pointed protomodular, it is locally algebraically cartesian closed if and only if the change of base functors along the initial maps have a right adjoint}. 

It is worth translating in detail what this means, and, rather surprisingly, we shall observe that this means a kind of extended dual of the action representativity. So let $\mathbb C$ be a locally algebraically cartesian closed pointed protomodular category. Let $Y$ be an object of $\mathbb C$ and  $\alpha_Y:1\rightarrowtail Y$ its associated initial map. We shall denote $\ss_Y$ by the right adjoint of $\alpha_Y^*$. Starting with any object $T$ in $\mathbb C$, the object $\ss_Y(T)$ is a split epimorphism above $Y$ which is equipped with a (universal) map from its kernel towards $T$. In other words it produces a universal split exact sequence we shall denote this way:
$$
\xymatrix@=20pt{
{\L(Y,T)\;} \ar@{>->}[rr]^{\kappa_T^Y} \ar[d]_{\l_T^Y} && Y\ltimes \L(Y,T) \ar@{->>}[rr]^{\psi_T^Y} && {\;Y} \ar@{>->}@<1ex>[ll]^{\zeta_T^Y}\\
T
              }
$$
which from any given similar situation, i.e. a split exact sequence with codomain $Y$ and a comparison map $h$:
$$
\xymatrix@=20pt{
{K\;} \ar@{>->}[rr]^{k} \ar@(l,l)[dd]_h \ar@{.>}[d]_{\bar h} && X  \ar@{->>}[rr]^{f} \ar@{.>}[d]_{Y\ltimes \bar h}  && Y \ar@{=}[d] \ar@{>->}@<1ex>[ll]^{s}\\
{\L(Y,T)\;} \ar@{>->}[rr]^{\kappa_T^Y} \ar[d]_{\l_T^Y} && Y\ltimes \L(Y,T) \ar@{->>}[rr]^{\psi_T^Y} && {\;Y} \ar@{>->}@<1ex>[ll]^{\zeta_T^Y}\\
T
              }
$$
produces a unique dotted factorization $\bar h$. In particular the following upper canonical split exact sequence:
$$
\xymatrix@=20pt{
{T\;} \ar@{>->}[rr]^{\iota_T} \ar@(l,l)[dd]_{1_T} \ar@{.>}[d]_{\S_T^Y} && Y\times T  \ar@{->>}[rr]^{p_Y} \ar@{.>}[d]_{Y\ltimes \S_T^Y}  && Y \ar@{=}[d] \ar@{>->}@<1ex>[ll]^{\iota_Y}\\
{\L(Y,T)\;} \ar@{>->}[rr]^{\kappa_T^Y} \ar[d]_{\l_T^Y} && Y\ltimes \L(Y,T) \ar@{->>}[rr]^{\psi_T^Y} && {\;Y} \ar@{>->}@<1ex>[ll]^{\zeta_T^Y}\\
T
              }
$$
determines a factorization which will be denoted by $\S_T^Y$. Starting from the more specific one with the diagonal $s_0$ as section, we have also the following factorization:
$$
\xymatrix@=20pt{
{Y\;} \ar@{>->}[rr]^{\iota_1} \ar@(l,l)[dd]_{1_Y} \ar@{.>}[d]_{\varpi_Y} && Y\times Y  \ar@{->>}[rr]^{p_0} \ar@{.>}[d]_{Y\ltimes \varpi_Y}  && Y \ar@{=}[d] \ar@{>->}@<1ex>[ll]^{s_0}\\
{\L(Y,Y)\;} \ar@{>->}[rr]^{\kappa_Y^Y} \ar[d]_{\l_Y^Y} && Y\ltimes \L(Y,Y) \ar@{->>}[rr]^{\psi_Y^Y} && {\;Y} \ar@{>->}@<1ex>[ll]^{\zeta_Y^Y}\\
Y
              }
$$
which we shall analyse below more precisely in the category $Gp$.

\subsection{The category $Gp$ of groups}

We shall explore in detail here the very unusual constructions involved in the local exponentiation property of the category $Gp$. In this case, we have $\L(Y,T)=\mathcal F(\underline Y,T)$, namely $\L(Y,T)$ is the set of applications from the underlying set of the group $Y$ to the underlying set of the group $T$ equipped with the group structure determined by the group structure on $T$. The action of the group $Y$ on this group $\mathcal F(\underline Y,T)$ associates with the pair $(y,\phi)$ the application: $$\phi\circ\tau_y:Y\rightarrow T\; ; \;\;\; z\mapsto \phi(z.y)$$
where $\tau_y$ is the translation on the right in the group $Y$ (in other words we get: $(^y\phi)(z)=\phi(z.y)$). So, in the category $Gp$, the parallelism between cartesian closedness and algebraic cartesian closedness is not only simply formal, but a kind of strong memory of the underlying exponentiation in $Set$. 

The homomorphism $\l_T^Y:\mathcal F(\underline Y,T)\rightarrow T$ is the evaluation at the unit element of $Y$. Given any split extension with a map $h$:
$$
\xymatrix@=20pt{
{K\;} \ar@{>->}[rr]^{} \ar[d]_{h} && Y\ltimes_{\alpha} K \ar@{->>}[rr]^{} && {\;Y} \ar@{>->}@<1ex>[ll]^{}\\
T
              }
$$
the group homomorphism $\bar h:K\rightarrow \mathcal F(\underline Y,T)$ is then defined by $\bar h(k)(y)=h(^yk)$.
In particular we get the group homomorphism $\S_T^Y: T\rightarrow \L(Y,T)$ defined by $\S_T^Y(t)(y)=t$, in other words $\S_T^Y(t)$ is the application constant on $t$. And also we get $\varpi_Y:Y\rightarrow \L(Y,Y)$ defined by $\varpi_Y(y)(z)=z.y.z^{-1}$ which is a very awkward way to integrate the ``inner action'' of $Y$ inside the category $Gp$.

If we start from:
$$
\xymatrix@=20pt{
{K\;} \ar@{>->}[rr]^{} \ar[d]_{h} && X \ar@{->>}[rr]^{f} && {\;Y} \ar@{>->}@<1ex>[ll]^{s}\\
T
              }
$$
we get: $(Y\ltimes \bar h)(x)=(f(x),\bar h(x.s\circ f(x^{-1})))$.

\subsection{First consequences of local algebraic cartesian closedness}

In this section, we shall investigate two important consequences of local algebraic cartesian closedness, namely strong protomodularity and peri-abelianness. These well identified properties in the category $Gp$ now clearly appear to have originated from locally algebraic cartesian closedness.

\subsection{Normal functor and strong protomodularity}

Recall that a left exact functor $U:\mathbb C\rightarrow D$ is called \emph{normal} when it is conservative and it reflects the normal monomorphisms. A protomodular category $\mathbb C$ is said to be \emph{strongly protomodular} \cite{BB} when any change of base functor $h^*:Pt_Y\mathbb C\rightarrow Pt_X\mathbb C$ with respect to the fibration of points is not only conservative but also normal. The categories $Gp$ of groups and $R$-$Lie$ of Lie $R$-algebras, for any commmutative ring $R$, are strongly protomodular. In this section we shall show that, when a protomodular category $\mathbb C$ is locally algebraically cartesian closed, it is necessarily strongly protomodular. Let us begin by the following observation:

\begin{lemma}
Let $U:\mathbb C\rightarrow D$ be a left exact conservative functor. Suppose moreover that $\mathbb D$ is protomodular. If it has a right adjoint $G$, then $U$ is normal.
\end{lemma}
\proof
The right adjoint $G$ is left exact and consequently preserves the monomorphims and the equivalence relations. Now let $m:X'\rightarrowtail X$ be a monomorphism in $\mathbb C$ such that the monomorphism $U(m)$ is normal to the equivalence relation $R$ in $\mathbb D$, namely such that we have a discrete fibration in $\mathbb D$:
$$
\xymatrix@=20pt{
{U(X')\times U(X')\;}\ar@{>->}[r]^>>>>>>{\mu} \ar@<1ex>[d]^{p_1}\ar@<-1ex>[d]_{p_0} & R \ar@<1ex>[d]^{d_1}\ar@<-1ex>[d]_{d_0}\\
{U(X')\;} \ar@{>->}[r]_{U(m)} \ar@{.>}[u] & U(X) \ar@{.>}[u]
}             
$$
Since $U(X')\times U(X')=U(X'\times X')$, by adjunction we get a map $\bar{\mu}$ in $\mathbb C$ such that $\epsilon_R.U(\bar{\mu})=\mu$:
$$
\xymatrix@=20pt{
{X'\times X'\;}\ar@{>.>}[r]_{\beta} \ar@<2ex>@{>->}[rr]^{\bar{\mu}} \ar@<1ex>[d]^{p_1}\ar@<-1ex>[d]_{p_0} & T\ar[r]_{\bar{\eta}} \ar@<1ex>[d]^{d_1}\ar@<-1ex>[d]_{d_0} & G(R) \ar@<1ex>[d]^{G(d_1)}\ar@<-1ex>[d]_{G(d_0)}\\
{X'\;} \ar@{>->}[r]_{U(m)} \ar[u] & X \ar[r]_{\eta_X} \ar[u] & G.U(X) \ar[u]
}             
$$
which determines a morphism between the equivalence relations $\nabla_{X'}$ and $G(R)$. We shall set $T=\eta_X^{-1}(G(R))$ and denote by $\beta$ the induced factorization. We are going to show that $m$ is normal to the equivalence relation $T$ and that $U(T)\simeq R$. For that, consider the following diagram in $\mathbb D$:
$$
\xymatrix@=20pt{
{U(X')\times U(X')\;}\ar@{>.>}[r]_>>>>>{U(\beta)} \ar@<2ex>@{>->}[rr]^{U(\bar{\mu})} \ar@<1ex>[d]^{p_1}\ar@<-1ex>[d]_{p_0} \ar@(u,u)[rrr]^{\mu} & U(T)\ar[r]_{U(\bar{\eta})} \ar@<1ex>[d]^{U(d_1)}\ar@<-1ex>[d]_{U(d_0)} & U.G(R) \ar[r]_{\epsilon_{R}}\ar@<1ex>[d]^{}\ar@<-1ex>[d]_{} & R \ar@<1ex>[d]^{d_1}\ar@<-1ex>[d]_{d_0}\\
{U(X')\;} \ar@{>->}[r]_{U(m)} \ar[u] & U(X) \ar[r]_{U(\eta_X)} \ar[u] \ar@(d,d)[rr]_{1_{U(X)}} & U.G.U(X) \ar[u] \ar[r]_>>>>>{\epsilon_{U(X)}} & U(X) \ar[u]
}             
$$
The map $\gamma=\epsilon_R.U(\bar{\eta)}$ produces an inclusion $U(T)\subset R$ of equivalence relations. Since the whole diagram is a discrete fibration and $\gamma$ is a monomorphism, then the left hand side part of the diagram is a discrete fibration. Accordingly $U(m)$ is normal to the equivalence relation $U(T)$. Now, when $\mathbb D$ is protomodular, a monomorphism is normal to at most one equivalence relation, and $\gamma$ is necessarily an isomorphism. On the other hand, since $U$ is left exact and conservative it reflects the pullbacks, so that $m$ is normal to $T$ in $\mathbb C$.
\endproof 
Whence the following:
\begin{theo}
Let $\mathbb C$ be a protomodular category which is locally algebraically cartesian closed. Then $\mathbb C$ is strongly protomodular.
\end{theo}
\proof
Since $\mathbb C$ is protomodular and locally algebraically cartesian closed, so is any fibre $Pt_Y\mathbb C$. Moreover any change of base functor:$$h^*:Pt_Y\mathbb C\rightarrow Pt_X\mathbb C$$
is left exact and conservative since $\mathbb C$ is protomodular; it has a right adjoint since $\mathbb C$ is locally algebraically cartesian closed. By the previous lemma it is normal, and $\mathbb C$ is strongly protomodular.  
\endproof

Now suppose, in addition, that $\mathbb C$ is pointed. Being pointed and strongly protomodular, it is such that two equivalence relations $(R,S)$ centralize if and only if their associated normal subobjects $(I_R,I_S)$ commute, see \cite{B25}; in other words the category $\mathbb C$ is such that we have the so-called equation ``Smith=Huq''.

The assertion of the theorem above was mentioned by G. Janelidze during the CT 2010 conference in Genova, but in the much stricter context of semi-abelian categories, as an immediate consequence of Proposition 9 in \cite{Bor1}, which deals with preservation of colimits and cannot be used in our context.

\subsection{Peri-abelian categories}

When $\mathbb C$ is a regular strongly unital category with finite colimits, the inclusion functor $Ab\mathbb C \rightarrowtail \mathbb C$ from the full subcategory of abelian objects in $\mathbb C$ admits a left adjoint which is given by the cokernel of the diagonal $s_0:X\rightarrowtail X\times X$, or by the coequalizer of the pair $(\iota_0,\iota_1);X\rightrightarrows X\times X$. Recall now the following \cite{B15}:
\begin{defi}\label{peri}
We shall say that a finitely cocomplete, regular Mal'cev category $\mathbb D$ is peri-abelian when the change of base functor along any map $h:Y\rightarrow Y'$ with respect to the fibration of points preserves the associated abelian object.
\end{defi}
If $(AbPt)\mathbb D$ denotes the subcategory of the abelian objects in the fibres of the fibration of points, it is equivalent to saying that the reg-epi reflection $A_{()}$ is cartesian, i.e. it preserves the cartesian maps:
$$\xymatrix@=10pt
{
{(AbPt)\mathbb D\;\;} \ar[rdd] \ar@{>->}[rr]^{}&& Pt\mathbb D \ar[ldd] \ar@(u,u)[ll]_{A_{()}}\\
\\
 & \mathbb D 
}
$$
The categories $Gp$ of groups, $Rg$ of non unitary commutative rings and $\mathbb K$-$Lie$ of $\mathbb K$-Lie algebras are peri-abelian. The previous definition was introduced in \cite{B15} as a tool to produce some cohomology isomorphisms which hold in the Eilenberg-Mac Lane cohomology of groups and in the cohomology of $\mathbb K$-Lie algebras.

\begin{theo}
Let $\mathbb C$ be a finitely cocomplete regular Mal'cev category which is locally algebraically cartesian closed. Then $\mathbb C$ is peri-abelian.
\end{theo}
\proof
This is a straightforward consequence of the fact that the change of basefunctors $h^*$, having a right adjoint, preserve cokernels or coequalizers.  
\endproof 

\subsection{The non-pointed protomodular case}

We shall suppose here that the category is still protomodular, but no longer pointed. We shall show that the algebraic exponentiability of any split monomorphism implies the algebraic exponentiability of any of its retractions, and from that, in the efficiently regular context, of a large class of morphisms. This will be the consequence of a very general result:
\begin{prop}
\label{composite} Let $\mathbb{C}$ be a protomodular category.  For morphisms $f:X\rightarrow Y$ and $g:Y\rightarrow Z$ in $\mathbb{C}$, if $g.f$ and $f$ are algebraically exponentiable, then $g$ is algebraically exponentiable.
\end{prop}
\begin{proof}
Recall that for any morphism $p: E \rightarrow B$ in $\mathbb{C}$, $p^{*}$ preserves all limits and reflects isomorphisms.  Therefore if $p^{*}$ has a right adjoint, then by the dual of the Weak Tripleability Theorem \cite{maclane}, $p^{*}$ is comonadic. The result follows from the well-known \emph{adjoint functor lifting theorem} (see e.g. \cite{barrwells}) applied to the diagram of functors:
$$
\xymatrix@R=2.5pc@C=2.5pc{
 Pt_Z(\mathbb{C}) \ar[rr]^{g^{*}} \ar@<0.5ex>[dr]^{(g.f)^*} & & Pt_Y(\mathbb{C}) \ar@<0.5ex>[dl]^{f^*} \\
 &Pt_X\mathbb{C}\ar@<0.5ex>[ur]^{\Phi_{f}}\ar@<0.5ex>[ul]^{\Phi_{gf}}&
}
$$
in which $\Phi_{gf}$ and $\Phi_{f}$ are the right adjoints to the functors $(g.f)^*$ and $f^*$ respectively.
\end{proof}

Whence the following theorem:
\begin{theo}
Let $\mathbb C$ be a protomodular category. Suppose that any split monomorphism $s$ is algebraically exponentiable. Then any of its retractions $f$ is algebraically exponentiable. When moreover $\mathbb C$ is efficiently regular, any map $h:X\rightarrow Y$ is algebraically exponentiable provided that its domain $X$ has a global support.
\end{theo}
\proof
The first point is a straighforward consequence of the previous proposition. Moreover, given any map $h$ in $\mathbb C$, we get $h=p_Y.(1_X,h)$, where the map $(1_X,h)$ is a monomorphism split by $p_X$, and consequently the change of base functor along it admits a right adjoint by assumption. When $\mathbb C$ is regular and $X$ has global support, the map $p_Y$ is a regular epimorphism. Now when $\mathbb C$ is efficiently regular, the change of base functor along $p_Y$ admits a right adjoint by Proposition \ref{regexp}. 
\endproof

\subsection{Back to the pointed case}

The previous construction of the right adjoint obviously applies in the pointed case and gives an alternative description of the centralizers. First, let  $(f,s):X\rightleftarrows Y$ be any split epimorphism; we get a classifying map $\gamma_{(f,s)}$:
$$
\xymatrix@=20pt{
{K\;} \ar@{>->}[rr]^{k} \ar@(l,l)[dd]_{1_K} \ar@{.>}[d]_{\gamma_{(f,s)}} && X  \ar@{->>}[rr]^{f} \ar@{.>}[d]_{Y\ltimes \gamma_{(f,s)}}  && Y \ar@{=}[d] \ar@{>->}@<1ex>[ll]^{s}\\
{\L(Y,K)\;} \ar@{>->}[rr]_{\kappa_K^Y} \ar[d]_{\l_K^Y} && Y\ltimes \L(Y,K) \ar@{->>}[rr]^{\psi_K^Y} && {\;Y} \ar@{>->}@<1ex>[ll]^{\zeta_K^Y}\\
K
              }
$$
Then $\Phi_Y[f,s]$ is given by the following equalizer:
$$
\xymatrix@=30pt{
{\Phi_Y[f,s]\;} \ar@{>->}[r]^{} &  K \ar@<1ex>[r]^{\S_K^Y} \ar@<-1ex>[r]_{\gamma_{(f,s)}} & \L(Y,K)
              }
$$
In particular the centralizer of a subobject $u:Y\rightarrowtail X$ is given by the following equalizer:
$$
\xymatrix@=30pt{
{Z[u]\;} \ar@{>->}[r]^{\zeta_u} &  X \ar@<1ex>[r]^{\S_X^Y} \ar@<-1ex>[r]_{\gamma_{(p_Y,(1,u))}} & \L(Y,X)
              }
$$

Let us make explicit these constructions in the category $Gp$. First we have: $\gamma_{(f,s)}(k)(y)=s(y).k.s(y^{-1})$, and consequently: $\Phi_Y[f,s]=\{k\in K/ \forall y\in Y, \; k=s(y).k.s(y^{-1})\}$, and thus, of course, $Z[u]=\{x\in X/ \forall y\in Y, \; x=y.x.y^{-1}\}$.

\end{document}